\documentclass[11pt,a4paper,reqno]{amsart}

\usepackage[T1]{fontenc}
\usepackage{lmodern}
\usepackage{microtype}
\usepackage{mathtools,amssymb}
\usepackage[authoryear,round]{natbib}
\usepackage[a4paper,textwidth=138mm,textheight=223mm,centering,headsep=8mm]{geometry}

\allowdisplaybreaks[2]
\numberwithin{equation}{section}

\newtheoremstyle{oldplain}
  {6pt}{6pt}{\itshape}{} {\scshape}{.}{0.5em}{}
\newtheoremstyle{olddefinition}
  {6pt}{6pt}{\normalfont}{} {\scshape}{.}{0.5em}{}
\theoremstyle{oldplain}
\newtheorem{theorem}{Theorem}[section]
\newtheorem{lemma}[theorem]{Lemma}
\newtheorem{corollary}[theorem]{Corollary}
\theoremstyle{olddefinition}
\newtheorem{remark}[theorem]{Remark}

\newcommand{\R}{\mathbb R}
\newcommand{\E}{\mathbb E}
\newcommand{\Pp}{\mathbb P}
\newcommand{\one}{\mathbf 1}
\newcommand{\su}{\mathrm{su}}
\newcommand{\Unif}{\operatorname{Unif}}
\newcommand{\Bin}{\operatorname{Bin}}

\title[Dimension comparison for Student's statistic]
      {Dimension comparison for Student's statistic under symmetric unimodality}
\author{Jacopo Lenzi}
\address{Department of Biomedical and Neuromotor Sciences\\
University of Bologna\\
Via San Giacomo 12, 40126 Bologna, Italy}
\email{jacopo.lenzi2@unibo.it}
\thanks{ORCID: 0000-0003-2882-4223.}
\date{August 13, 2026}

\begin{document}

\begin{abstract}
Let $q_n(r)$ denote the tail probability at $r$ of the self-normalized sum of
$n$ independent centered uniform variables. At $r=3$, the first
distribution-sensitive term in the two-sided Edgeworth expansion of
Student's statistic vanishes. We evaluate the expansion at the common moving
boundary $r_n=3+\lambda/n$ in dimensions $n$ and $n-k$. Uniformly over
deletion ranks retaining a fixed positive fraction of observations, the first
nonzero difference converges to an explicit phase surface
$H_\delta(\lambda)$; its zero curve unifies fixed, sublinear and fixed-fraction
deletions, with tangent crossing $12/35$. Through the Khintchine scale-mixture
representation, this comparison yields a single compactly supported
$C^\infty$ symmetric unimodal parent, independent of $n$, whose Student tail
exceeds the equal-scale uniform tail for all sufficiently large $n$ along
nominal levels tending to $2\{1-\Phi(\sqrt3)\}$ from below. In contrast, a
quantile-ratio order shows that the uniform parent maximizes every even moment
and every convergent even power series with nonnegative coefficients. We also
derive the fixed-confidence dimension expansion and an exact reversal at
$n=6$.
\end{abstract}

\maketitle

\section{Introduction}\label{sec:introduction}

Let $X_1,\ldots,X_n$ be independent and identically distributed from a law
symmetric about a center $\theta$, and write
\begin{equation}
T_n=\frac{\sqrt n\,(\bar X_n-\theta)}{S_n},
\qquad
S_n^2=\frac1{n-1}\sum_{i=1}^n(X_i-\bar X_n)^2.
\label{eq:1.1}
\end{equation}
The parameter is the center of symmetry.  It agrees with the mean whenever
the latter exists, but the finite-sample event considered below requires no
moment assumption.  Let $\mathcal F_{\su}$ denote the class of nondegenerate
laws having a density symmetric about zero and nonincreasing on $[0,\infty)$.
The scale of a centered uniform law is immaterial for $T_n$.

The behavior of Student's statistic under nonnormal symmetric parents has a
long history.  Early expansions were obtained by \citet{chung1946}, and the
geometric and moment comparison under orthant symmetry was developed by
\citet{efron1969}.  The general validity theory for smooth functions of
sample means was established by \citet{bhattacharyaGhosh1978}.  Remote-tail
comparisons based on stretching and majorization were proved by
\citet{benjamini1983}, while \citet{hall1987} obtained the Edgeworth expansion
under minimal moment conditions.

\citet[Chapter~5]{basu1991} studied finite-sample Student coverage over the
broader class of symmetric unimodal laws. Using the Khintchine mixture
representation, \citet[Section~5.4.1, pp.~84--87]{basu1991} reduced the
problem to numerical minimization over independent centered uniforms with
unequal scales. Computations through $n=10$ identified simulation-based,
dimension-dependent thresholds, with selected values reported in Table~5.2,
at and above which the equal-scale uniform was identified numerically as the
minimizer. Below the corresponding threshold, Basu reported, on the basis of
that numerical investigation, that $p\delta_0+(1-p)\Unif[-1,1]$, for a
suitable $p$, has smaller central coverage than the uniform parent (p.~87).
In Section~5.4.2 (pp.~87--89), Hall's expansion gives the kurtosis-dependent
contribution $\kappa y(y^2-3)\phi(y)/(6n)$, whose sign changes at $y=\sqrt3$;
Table~5.4 compares the approximation with exact normal or simulated
nonnormal coverage values. A published account of the confidence-interval
analysis appears in \citet{basu1995}.
In the symmetric unimodal class, the stretching theorem of
\citet{benjamini1983} gives exact least favorability of the uniform parent
when the Student critical value satisfies $t\ge n-1$.
\citet{hendriks2006} later expressed the Edgeworth boundary in nominal-tail
form as
\begin{equation}
\alpha_0=2\bar\Phi(\sqrt3)=0.0832645\ldots,
\qquad \bar\Phi=1-\Phi,
\label{eq:1.2}
\end{equation}
and proposed uniform calibration throughout the range $0<\alpha<\alpha_0$.
If $c^U_{n,\alpha}$ is defined by
$\Pp_U\{|T_n|>c^U_{n,\alpha}\}=\alpha$, the corresponding finite-sample
assertion is
\begin{equation}
\Pp_F\{|T_n|>c^U_{n,\alpha}\}\le \alpha
\quad
(F\in\mathcal F_{\su},\ n\ge2,\ 0<\alpha<\alpha_0).
\label{eq:1.3}
\end{equation}
The exact stretching results concern a more remote range of critical values,
whereas Basu's reported numerical thresholds and atom-at-zero comparison
reach substantially smaller values. At $y=\sqrt3$ the first
distribution-sensitive correction vanishes. Evaluating the expansion at the
common boundary $r_n=3+\lambda/n$ in dimensions $n$ and $n-k$ produces a
two-parameter comparison: $\lambda$ measures displacement within the
critical layer and $k/n$ records the loss of dimension. The limiting phase
surface below gives the sign uniformly over deletion ranks that retain a
fixed positive fraction of the observations. Its zero curve contains the
fixed and sublinear tangent regime as well as the fixed-fraction regime, and
the Khintchine mixture transfer converts the atom-at-zero mechanism into one
fixed smooth parent.

Higher-order expansions for Studentized statistics were developed by
\citet{finner2010} and \citet{gerlovina2021}. The recent work of
\citet{beckedorf2025} supplies explicit self-normalized coefficients and
strong approximation bounds, while the validity theorem of
\citet{bhattacharyaGhosh1978} covers the moving two-sided sets used here.
These results provide the expansion input. Subtracting the expansions at two
dimensions while keeping the moving cone boundary common produces the phase
surface.

For $U_i\stackrel{\mathrm{i.i.d.}}{\sim}\Unif[-1,1]$, put
\begin{equation}
q_n(r)=\Pp\left\{
\frac{(U_1+\cdots+U_n)^2}{U_1^2+\cdots+U_n^2}>r
\right\}.
\label{eq:1.4}
\end{equation}
At $r_n=3+\lambda/n$ and deletion fraction $k/n\to\delta<1$, the leading
comparison is
\begin{equation}
\frac{n^3}{k}\{q_{n-k}(r_n)-q_n(r_n)\}
=H_{k/n}(\lambda)+o(1),
\label{eq:1.5}
\end{equation}
where
\begin{equation}
H_\delta(\lambda)
=\sqrt3\phi(\sqrt3)
\left\{-\frac{3\lambda}{10(1-\delta)}
+\frac{9(2-\delta)}{175(1-\delta)^2}\right\}.
\label{eq:1.6}
\end{equation}
The zero is
\begin{equation}
\lambda_c(\delta)=\frac6{35}\frac{2-\delta}{1-\delta}.
\label{eq:1.7}
\end{equation}
For fixed or sublinear deletion rank this gives $\lambda_c(0)=12/35$; for a
fixed retained fraction it gives the full curve.  Conditioning a two-scale
Khintchine mixture on the number of observations at its larger scale then
produces a single fixed smooth parent for which the inequality in (\ref{eq:1.3}) is
reversed for every sufficiently large $n$ along a sequence
$\alpha_n<\alpha_0$.

Two further results place the tail reversal within a broader order
structure. Inverting the same expansion at fixed confidence gives a separate
dimension phase diagram. The usual two-sided $5\%$ calibration has limiting
cone parameter $\{\Phi^{-1}(0.975)\}^2=3.841458\ldots$ and is therefore
asymptotically separated from the $r=3+O(n^{-1})$ layer in (\ref{eq:1.5}).
Independently, Efron's moment argument can be combined with a quantile-ratio
order: the uniform parent maximizes every even moment and every convergent
even power series with nonnegative coefficients within $\mathcal F_{\su}$.
Thus the reversal coexists with exact uniform extremality for a cone of
analytic even functionals, whereas indicator tails lie outside that cone.

The main arguments are analytic. Appendix~\ref{app:n6} gives an independent
finite reversal at $n=6$.

\section{Cone probabilities and scale mixtures}\label{sec:cone}

Put
\begin{equation}
Z_n=\frac{\sum_{i=1}^n(X_i-\theta)}
          {\{\sum_{i=1}^n(X_i-\theta)^2\}^{1/2}}.
\label{eq:2.1}
\end{equation}
Whenever the denominator is nonzero,
\begin{equation}
T_n^2=\frac{(n-1)Z_n^2}{n-Z_n^2}.
\label{eq:2.2}
\end{equation}
If all observations are equal, we define rejection to be false.  This event
has probability zero under every density in $\mathcal F_{\su}$.
For $0<r<n$, let
\begin{equation}
C_{n,r}=\left\{x\in\R^n:
\left(\sum_{i=1}^n x_i\right)^2
>r\sum_{i=1}^n x_i^2\right\}.
\label{eq:2.3}
\end{equation}
For independent $U_i\sim\Unif[-1,1]$, define
\begin{equation}
Q_{n,r}(a_1,\ldots,a_n)
=\Pp\{(a_1U_1,\ldots,a_nU_n)\in C_{n,r}\},
\qquad
q_n(r)=Q_{n,r}(1,\ldots,1).
\label{eq:2.4}
\end{equation}
The Student threshold $t$ and the cone parameter are related by
\begin{equation}
r=\frac{nt^2}{n-1+t^2},
\qquad
t^2=\frac{(n-1)r}{n-r}.
\label{eq:2.5}
\end{equation}
This cone formulation is classical; the uniform case was studied, for
example, by \citet{perlo1933}, and related geometry appears in
\citet{hotelling1961}.

Every law in $\mathcal F_{\su}$ admits the Khintchine representation
\begin{equation}
X\stackrel d=RU,
\label{eq:2.6}
\end{equation}
where $U$ is centered uniform, $R\ge0$, and $R$ and $U$ are independent; see
\citet[Chapter~1]{dharmadhikari1988}.  Conditional on independent copies of
$R$, the cone probability is (\ref{eq:2.4}).  This is also the scale-mixture coordinate
system used by \citet[Chapter~5]{basu1991}.  At zero lower scale, the two-point
mixing law below is precisely the atom-at-zero perturbation used there.  The
next lemma identifies its first variation at zero mixing weight and transfers
a strict gain to a smooth density.

\begin{lemma}[first variation at a vanishing scale]
Suppose $n\ge2$, $0<r<n$, and $q_{n-1}(r)>q_n(r)$.  Then there is a compactly
supported $C^\infty$ density $f$, symmetric and nonincreasing on
$[0,\infty)$, such that
\begin{equation}
\Pp_f(C_{n,r})>q_n(r).
\label{eq:2.7}
\end{equation}
\end{lemma}

\begin{proof}
Let $R_{p,\varepsilon}$ equal $\varepsilon$ with probability $p$ and one with
probability $1-p$, independently of $U$, and put $X=R_{p,\varepsilon}U$.
The density of $X$ is
\begin{equation}
f_{p,\varepsilon}(x)
=\frac{p}{2\varepsilon}\one_{\{|x|<\varepsilon\}}
 +\frac{1-p}{2}\one_{\{|x|<1\}}.
\label{eq:2.8}
\end{equation}
It is symmetric and nonincreasing in $|x|$.  Conditioning on the number of
observations from the smaller scale gives
\begin{equation}
P_{p,\varepsilon}(r)
=\sum_{j=0}^n\binom{n}{j}p^j(1-p)^{n-j}
Q_{n,r}(\underbrace{\varepsilon,\ldots,\varepsilon}_{j},1,\ldots,1).
\label{eq:2.9}
\end{equation}
Except when $(n,r)=(2,1)$, the restriction of the cone boundary to the
coordinate subspace obtained by setting one coordinate to zero is the zero
set of a nonzero quadratic form and therefore has Lebesgue measure zero.
The exceptional pair cannot satisfy the premise, since
$q_1(1)=0<q_2(1)=1/2$.  Hence in every case covered by the premise,
\begin{equation}
Q_{n,r}(\varepsilon,1,\ldots,1)\longrightarrow q_{n-1}(r)
\qquad(\varepsilon\downarrow0).
\label{eq:2.10}
\end{equation}
Differentiating (\ref{eq:2.9}) at $p=0$ gives
\begin{equation}
\left.\frac{\partial}{\partial p}P_{p,\varepsilon}(r)\right|_{p=0}
=n\{Q_{n,r}(\varepsilon,1,\ldots,1)-q_n(r)\}>0
\label{eq:2.11}
\end{equation}
for all sufficiently small $\varepsilon$.  A sufficiently small positive $p$
then proves (\ref{eq:2.7}) for the density (\ref{eq:2.8}).

Finally convolve (\ref{eq:2.8}) with a compactly supported $C^\infty$ symmetric
decreasing approximate identity.  Convolution preserves symmetric decrease,
and product total variation tends to zero with the one-dimensional
$L^1$ distance.  The strict inequality therefore persists.
\end{proof}

\section{Dimension comparison near the sign-change boundary}\label{sec:dimension}

Rescale the uniforms to variance one and write
\begin{equation}
X_i\sim\Unif[-\sqrt3,\sqrt3],
\qquad
Z_n=\frac{\sum_iX_i}{(\sum_iX_i^2)^{1/2}}.
\label{eq:3.1}
\end{equation}
The common scale has no effect on $Z_n$. Both probabilities in the comparison
below are evaluated at the boundary indexed by the larger dimension $n$.
This common-boundary convention places fixed, sublinear and fixed-fraction
deletion ranks in one asymptotic statement.

\begin{theorem}[dimension comparison]
Fix $\lambda\in\R$ and put
\begin{equation}
r_n=3+\frac\lambda n.
\label{eq:3.2}
\end{equation}
Then
\begin{equation}
q_n(r_n)
=\alpha_0-\frac{\phi(\sqrt3)\lambda}{\sqrt3\,n}
+O(n^{-2}).
\label{eq:3.3}
\end{equation}
Let
\begin{equation}
a=-\frac{3\sqrt3\phi(\sqrt3)}{10},
\qquad
b=\frac{9\sqrt3\phi(\sqrt3)}{175},
\label{eq:3.4}
\end{equation}
and, for $0\le\delta<1$,
\begin{equation}
H_\delta(\lambda)
=\frac{a\lambda}{1-\delta}
+\frac{b(2-\delta)}{(1-\delta)^2}.
\label{eq:3.5}
\end{equation}
For every $\gamma_0\in(0,1)$,
\begin{equation}
\sup_{1\le k\le(1-\gamma_0)n}
\left|
\frac{n^3}{k}\{q_{n-k}(r_n)-q_n(r_n)\}
-H_{k/n}(\lambda)
\right|\longrightarrow0.
\label{eq:3.6}
\end{equation}
The convergence is locally uniform in $\lambda$.  The unique zero of the
limiting affine function is
\begin{equation}
\lambda_c(\delta)=\frac6{35}\frac{2-\delta}{1-\delta}.
\label{eq:3.7}
\end{equation}
Thus, whenever $1\le k_n<n$ eventually and $k_n/n\to\delta<1$, the difference
$q_{n-k_n}(r_n)-q_n(r_n)$ is eventually positive for
$\lambda<\lambda_c(\delta)$ and eventually negative for
$\lambda>\lambda_c(\delta)$.
\end{theorem}

\begin{corollary}[tangent and fixed-fraction limits]
If $1\le k_n=o(n)$, then
\begin{equation}
\frac{n^3}{k_n}\{q_{n-k_n}(r_n)-q_n(r_n)\}
\longrightarrow
\sqrt3\phi(\sqrt3)
\left(\frac{18}{175}-\frac{3\lambda}{10}\right).
\label{eq:3.8}
\end{equation}
If $m_n/n\to\gamma\in(0,1)$, then
\begin{equation}
\begin{split}
n^2\{q_{m_n}(r_n)-q_n(r_n)\}
\longrightarrow{}&
\sqrt3\phi(\sqrt3)
\left\{-\frac{3\lambda}{10}(\gamma^{-1}-1)
+\frac9{175}(\gamma^{-2}-1)\right\}.
\end{split}
\label{eq:3.9}
\end{equation}
The corresponding zeros are $12/35$ and $6(1+\gamma)/(35\gamma)$.
\end{corollary}

\begin{corollary}[crossings]
Fix $\gamma_0\in(0,1)$.  For all sufficiently large $n$ and every integer
$1\le k\le(1-\gamma_0)n$, there is a point $r_{n,k}^*$ such that
\begin{equation}
q_{n-k}(r_{n,k}^*)=q_n(r_{n,k}^*)
\label{eq:3.10}
\end{equation}
and the points may be selected so that
\begin{equation}
\sup_{1\le k\le(1-\gamma_0)n}
\left|n(r_{n,k}^*-3)-\lambda_c(k/n)\right|\longrightarrow0.
\label{eq:3.11}
\end{equation}
Only existence is asserted.
\end{corollary}

\begin{proof}
For each $n$ and $k$, the difference $q_{n-k}(r)-q_n(r)$ is continuous near
$r=3$.  Fix $\eta>0$.  Since $\lambda_c$ is bounded on
$[0,1-\gamma_0]$, apply (\ref{eq:3.6}) at
\[
r=3+\frac{\lambda_c(k/n)-\eta}{n}
\quad\hbox{and}\quad
r=3+\frac{\lambda_c(k/n)+\eta}{n}.
\]
The limiting values have opposite signs uniformly in $k$, because
$\partial H_\delta(\lambda)/\partial\lambda=a/(1-\delta)<0$ and its absolute
value is bounded away from zero.  The intermediate value theorem gives a
zero between the two points.  Choose a zero nearest
$3+\lambda_c(k/n)/n$ in the interval obtained with $\eta=1$.  Repeating the
argument for every fixed $0<\eta<1$ proves (\ref{eq:3.11}).
\end{proof}

\subsection{Conditional cumulants}

Write $X_i=\varepsilon_i\sqrt{V_i}$, where the signs are independent
Rademacher variables and $V_i=X_i^2$.  Put
\begin{equation}
W_i=\frac{V_i}{\sum_jV_j},
\qquad
A_j=\sum_iW_i^j.
\label{eq:3.12}
\end{equation}
Conditionally on the magnitudes,
\begin{equation}
Z_n=\sum_i\varepsilon_i\sqrt{W_i}.
\label{eq:3.13}
\end{equation}
Since the conditional mean is zero and the conditional variance is one, the
fourth and sixth unconditional cumulants are
\begin{equation}
\kappa_{4,n}=-2\E A_2,
\qquad
\kappa_{6,n}=16\E A_3.
\label{eq:3.14}
\end{equation}
For the variance-one uniform law,
\begin{equation}
\E V_i=1,
\qquad
\E V_i^2=\frac95,
\qquad
\E V_i^3=\frac{27}{7}.
\label{eq:3.15}
\end{equation}
Let $D_n=\sum_i(V_i-1)$.  On $\{|D_n|\le n/2\}$, expansion of the relevant
negative powers of $n+D_n$ gives, with $\mu_j=\E V_i^j$,
\begin{equation}
\begin{split}
\E A_2
&=\frac1{n^2}\E\left[\sum_iV_i^2
\left\{1-\frac{2D_n}{n}+\frac{3D_n^2}{n^2}\right\}\right]
+O(n^{-3})\\
&=\frac{\mu_2}{n}
+\frac{-2(\mu_3-\mu_2)+3\mu_2(\mu_2-1)}{n^2}
+O(n^{-3}).
\end{split}
\label{eq:3.16}
\end{equation}
The complementary event is exponentially small because $0\le V_i\le3$.
The same ratio expansion yields
\begin{equation}
\E A_2=\frac{9}{5n}+\frac{36}{175n^2}+O(n^{-3}),
\quad
\E A_3=\frac{27}{7n^2}+O(n^{-3}),
\quad
\E(A_2^2)=\frac{81}{25n^2}+O(n^{-3}).
\label{eq:3.17}
\end{equation}
Consequently,
\begin{equation}
\kappa_{4,n}=-\frac{18}{5n}-\frac{72}{175n^2}+O(n^{-3}),
\qquad
\kappa_{6,n}=\frac{432}{7n^2}+O(n^{-3}).
\label{eq:3.18}
\end{equation}

\subsection[The expansion at r=3]{The expansion at $r=3$}

The conditional characteristic function has the exact product form
\[
\E(e^{itZ_n}\mid W)=\prod_{i=1}^n\cos(t\sqrt{W_i}),
\]
and hence
\begin{equation}
\log\E(e^{itZ_n}\mid W)
=-\frac{t^2}{2}-\frac{t^4}{12}A_2-\frac{t^6}{45}A_3+O_t(A_4).
\label{eq:3.19}
\end{equation}
On $\{\sum_iV_i\ge n/2\}$, $\max_iW_i\le6/n$; its complement is
exponentially unlikely.  Since $\E A_4$, $\E(A_2A_3)$, and $\E A_2^3$ are
$O(n^{-3})$, exponentiation gives
\begin{equation}
\E e^{itZ_n}=e^{-t^2/2}\left\{
1-\frac{t^4}{12}\E A_2-\frac{t^6}{45}\E A_3
+\frac{t^8}{288}\E(A_2^2)+O_t(n^{-3})\right\}.
\label{eq:3.20}
\end{equation}
Formal inversion therefore identifies, for $x$ in a fixed compact subset of
$(0,\infty)$,
\begin{equation}
\begin{split}
\Pp\{|Z_n|>x\}
={}&2\bar\Phi(x)
 +\frac{\kappa_{4,n}}{12}\phi(x)H_3(x)\\
&+\frac{\kappa_{6,n}}{360}\phi(x)H_5(x)
 +\frac{\kappa_{4,n}^2}{576}\phi(x)H_7(x)
 +O(n^{-3}),
\end{split}
\label{eq:3.21}
\end{equation}
where $H_j$ is the probabilists' Hermite polynomial.  In the last displayed
term, replacing $\E(A_2^2)$ by $(\E A_2)^2$ changes the expansion only by
$O(n^{-3})$.

We next record the remainder used in the comparison.  Apply Theorem~2(a) and
(1.15) of \citet{bhattacharyaGhosh1978}, together with its correction
\citep{bhattacharyaGhosh1980}, with $s=8$, to the smooth function
\begin{equation}
H(u,v)=u/\sqrt v
\label{eq:3.22}
\end{equation}
of the sample means of $(X,X^2)$. Choose a globally $C^8$ extension that
agrees with $H$ on a fixed neighborhood of $(0,1)$. Since the uniform law is
bounded, the sample means leave that neighborhood with probability
$O(e^{-cn})$ for some $c>0$, which is negligible at the order used below.
Its density is positive on the interior of its support, the functions
$1,x,x^2$ are linearly independent there, $H$ is $C^8$ near $(0,1)$, and
$\nabla H(0,1)=(1,0)$.  If $\psi_{8,n}$ denotes their formal Edgeworth
density, their theorem gives
\begin{equation}
\sup_{D\in\mathcal B(\R)}
\left|\Pp\{Z_n\in D\}-\int_D\psi_{8,n}(z)\,dz\right|
=o(n^{-3}).
\label{eq:3.23}
\end{equation}
Thus moving two-sided tail sets are covered.  Through order $n^{-3}$ the even
conditional characteristic function involves only fixed monomials in
$A_2,A_3,A_4$; their expectations have expansions in integer powers of
$n^{-1}$ by the same ratio calculation as in (\ref{eq:3.16}).  All half-integer terms
therefore vanish.  More generally, for every fixed compact set $K\subset(0,\infty)$, uniformly for $r\in K$, we may write
\begin{equation}
q_j(r)=N(r)+\frac{A(r)}j+\frac{B(r)}{j^2}
      +\frac{C(r)}{j^3}+R_j(r),
\qquad
\sup_{r\in K}|R_j(r)|=o(j^{-3}),
\label{eq:3.24}
\end{equation}
where
\begin{equation}
N(r)=2\bar\Phi(\sqrt r),
\qquad
A(r)=-\frac3{10}\phi(\sqrt r)H_3(\sqrt r).
\label{eq:3.25}
\end{equation}
Until Section~\ref{sec:fixed-confidence} only a compact neighborhood of three is used.
The explicit coefficients in \citet{beckedorf2025}, specialized to the
symmetric uniform law, agree with (\ref{eq:3.21}).  In particular, at $x=\sqrt3$,
\begin{equation}
H_3(x)=0,
\qquad
H_5(x)=-6\sqrt3,
\qquad
H_7(x)=48\sqrt3,
\label{eq:3.26}
\end{equation}
and substitution of (\ref{eq:3.18}) gives
\begin{equation}
B(3)=\frac{9\sqrt3\phi(\sqrt3)}{175}.
\label{eq:3.27}
\end{equation}
This use of \citet{beckedorf2025} concerns the coefficient formulas.  The
uniformity over deletion ranks below is obtained from the one-dimensional
remainder in (\ref{eq:3.24}).

\subsection{Proof of the dimension comparison}

Differentiating (\ref{eq:3.25}) at $r=3$ gives
\begin{equation}
A'(3)=-\frac{3\sqrt3\phi(\sqrt3)}{10},
\qquad
N'(3)=-\frac{\phi(\sqrt3)}{\sqrt3}.
\label{eq:3.28}
\end{equation}
Put $m=n-k$ and $\delta=k/n$.  Exact denominator algebra in (\ref{eq:3.24}) gives
\begin{equation}
\begin{split}
\frac{n^3}{k}\{q_{n-k}(r_n)-q_n(r_n)\}
={}&\frac{nA(r_n)}{1-\delta}
+B(r_n)\frac{2-\delta}{(1-\delta)^2}\\
&+\frac{C(r_n)}n
\frac{3-3\delta+\delta^2}{(1-\delta)^3}
+\frac{n^3}{k}\{R_{n-k}(r_n)-R_n(r_n)\}.
\end{split}
\label{eq:3.29}
\end{equation}
Taylor expansion gives
\begin{equation}
nA(r_n)=a\lambda+O(n^{-1}),
\qquad
B(r_n)=b+O(n^{-1}).
\label{eq:3.30}
\end{equation}
If $1\le k\le(1-\gamma_0)n$, then $m\ge\gamma_0n$.  The $C$ term is
$O(n^{-1})$ uniformly.  To treat the remainder, put
\[
\varepsilon_N=
\sup_{j\ge\gamma_0N}\,j^3\sup_{r\in K}|R_j(r)|,
\]
where $K$ is a fixed compact neighborhood of three.  Then
$\varepsilon_N\to0$ and
\[
\sup_{1\le k\le(1-\gamma_0)n}
\frac{n^3}{k}|R_{n-k}(r_n)-R_n(r_n)|
\le (\gamma_0^{-3}+1)\varepsilon_n.
\]
The first two terms in (\ref{eq:3.29}) differ from $H_\delta(\lambda)$ by $O(n^{-1})$
uniformly for $0\le\delta\le1-\gamma_0$ and locally uniformly in $\lambda$.
This proves (\ref{eq:3.6}).  Solving $H_\delta(\lambda)=0$ gives (\ref{eq:3.7}), while (\ref{eq:3.3})
follows from $A(3)=0$ and the expansion of $N(r_n)$.

If $k_n=o(n)$, let $k_n/n\to0$ in (\ref{eq:3.6}) to obtain (\ref{eq:3.8}).  If
$m_n/n\to\gamma\in(0,1)$, put $k_n=n-m_n$ and multiply (\ref{eq:3.6}) by
$k_n/n\to1-\gamma$; simplification gives (\ref{eq:3.9}).
\qed

\begin{remark}[the excluded retained-dimension boundary]
If $m_n\to\infty$ and $m_n/n\to0$, the same expansion gives
\begin{equation}
m_n^2\{q_{m_n}(r_n)-q_n(r_n)\}
\longrightarrow \frac{9\sqrt3\phi(\sqrt3)}{175}.
\label{eq:3.31}
\end{equation}
The condition $m_n\to\infty$ is essential: a fixed retained dimension is a
different boundary regime.
\end{remark}

\subsection{One parent for all sufficiently large dimensions}

\begin{corollary}[fixed parent]
Fix $\gamma\in(0,1)$ and
\begin{equation}
0<\lambda<\frac6{35}\frac{1+\gamma}{\gamma}.
\label{eq:3.32}
\end{equation}
There exist a compactly supported $C^\infty$ density $f$, independent of $n$,
symmetric and nonincreasing on $[0,\infty)$, and a constant
$d_f(\lambda)>0$ such that, with
\begin{equation}
r_n=3+\frac\lambda n,
\qquad
t_n^2=\frac{(n-1)r_n}{n-r_n},
\qquad
\alpha_n=q_n(r_n),
\label{eq:3.33}
\end{equation}
one has
\begin{equation}
\Pp_f\{|T_n|>t_n\}-\alpha_n
=\frac{d_f(\lambda)}{n^2}+o(n^{-2})>0
\label{eq:3.34}
\end{equation}
for every sufficiently large $n$, while $\alpha_n<\alpha_0$.
\end{corollary}

\begin{proof}
First consider the singular scale law $X_0=R_0U$, where
$U\sim\Unif[-\sqrt3,\sqrt3]$,
$\Pp(R_0=1)=\gamma$, and $\Pp(R_0=0)=1-\gamma$.  Conditional on the number
$M_n\sim\Bin(n,\gamma)$ of nonzero observations, its tail is
$q_{M_n}(r_n)$.  Put $q_m(r)=0$ when $r\ge m$, and $q_0(r)=0$ under the
all-zero convention.  Binomial concentration and the uniform comparison
(\ref{eq:3.6}) give
\begin{equation}
n^2\{\E q_{M_n}(r_n)-q_n(r_n)\}
\longrightarrow (1-\gamma)H_{1-\gamma}(\lambda)>0.
\label{eq:3.35}
\end{equation}
Indeed, (\ref{eq:3.6}) is uniform when $M_n/n$ lies in a fixed neighborhood of
$\gamma$, while the complementary occupancies have exponentially small
probability.

We next replace the atom by a positive scale.  For a symmetric variance-one
parent $F$ with bounded density and sufficiently many moments, write
$\mu_j=\E_FX^j$.  The explicit symmetric coefficients of
\citet{beckedorf2025} give
\begin{equation}
A_F'(3)=-\frac{\sqrt3\phi(\sqrt3)}6\mu_4,
\qquad
B_F(3)=\sqrt3\phi(\sqrt3)
\left(\frac{\mu_4^2}{3}-\frac{4\mu_6}{15}\right).
\label{eq:3.36}
\end{equation}
Moreover $A_F(3)=0$, since the first coefficient contains the factor
$H_3(\sqrt3)$.  Their remainder, for the compactly supported densities used
below, is $o(n^{-2})$ uniformly on a fixed neighborhood of $r=3$.  Relative
to the uniform parent,
\begin{equation}
\begin{split}
d_F(\lambda)=\sqrt3\phi(\sqrt3)\biggl\{
&-\frac\lambda6\left(\mu_4-\frac95\right)
+\frac{\mu_4^2}{3}-\frac{4\mu_6}{15}-\frac9{175}
\biggr\},
\end{split}
\label{eq:3.37}
\end{equation}
and
\begin{equation}
\Pp_F\{Z_n^2>r_n\}-q_n(r_n)
=\frac{d_F(\lambda)}{n^2}+o(n^{-2}).
\label{eq:3.38}
\end{equation}

Let $R_\varepsilon$ equal one with probability $\gamma$ and $\varepsilon$
with probability $1-\gamma$, and put $X_\varepsilon=R_\varepsilon U$.
After standardization, its even moments are
\begin{equation}
\mu_{2j}(\varepsilon)
=\frac{3^j}{2j+1}
\frac{\gamma+(1-\gamma)\varepsilon^{2j}}
     {\{\gamma+(1-\gamma)\varepsilon^2\}^{j}}.
\label{eq:3.39}
\end{equation}
At $\varepsilon=0$, substitution of
$\mu_4=9/(5\gamma)$ and $\mu_6=27/(7\gamma^2)$ in (\ref{eq:3.37}) gives
$(1-\gamma)H_{1-\gamma}(\lambda)$.  By (\ref{eq:3.32}) this is positive.  Continuity
therefore gives $d_{F_\varepsilon}(\lambda)>0$ for every sufficiently small
fixed $\varepsilon>0$.  Such a law has a compactly supported symmetric
decreasing two-level density, and (\ref{eq:3.38}) holds for this one fixed parent.

Convolve the density with a compactly supported $C^\infty$ symmetric
decreasing approximate identity.  The standardized moments in (\ref{eq:3.37}) vary
continuously with the convolution width, so one sufficiently small fixed
width preserves $d_F(\lambda)>0$.  Finally (\ref{eq:3.3}) gives
$\alpha_n<\alpha_0$, and (\ref{eq:2.2})--(\ref{eq:2.5}) give (\ref{eq:3.34}).
\end{proof}

\section{Even moments and analytic transforms}\label{sec:moments}

Fix $n\ge2$. For a nondegenerate symmetric law $F$ with no atom at zero, let $a$ be the
quantile function of $|X|$ under $F$.  For an i.i.d.\ sample from $F$, write
\begin{equation}
Z_F=\frac{X_1+\cdots+X_n}{(X_1^2+\cdots+X_n^2)^{1/2}}.
\label{eq:4.1}
\end{equation}
The next result packages Efron's Rademacher moment comparison in a form
adapted to symmetric scale laws. It identifies a broad analytic class for
which the uniform law remains exactly extremal, even though the indicator-tail
order can reverse.

\begin{theorem}[quantile-ratio moment comparison]
Let $F$ and $G$ be nondegenerate symmetric laws with no atom at zero, and let
$a$ and $b$ be the quantile functions of their absolute values.  Suppose
$a(u)/b(u)$ is nondecreasing on $(0,1)$.  Then, for every integer $m\ge0$,
\begin{equation}
\E Z_F^{2m}\le \E Z_G^{2m}.
\label{eq:4.2}
\end{equation}
If $a/b$ is not almost everywhere constant, the inequality is strict for
every $m\ge2$.  Consequently, whenever $c_m\ge0$ and
$\sum_{m\ge0}c_mx^{2m}$ converges on $[-\sqrt n,\sqrt n]$,
\begin{equation}
\E\sum_{m\ge0}c_mZ_F^{2m}
\le
\E\sum_{m\ge0}c_mZ_G^{2m}.
\label{eq:4.3}
\end{equation}
\end{theorem}

\begin{proof}
On a common probability space, write
\begin{equation}
X_i^F=\varepsilon_i a(V_i),
\qquad
X_i^G=\varepsilon_i b(V_i),
\label{eq:4.4}
\end{equation}
where the $V_i$ are i.i.d.\ uniform on $(0,1)$ and the $\varepsilon_i$ are
independent Rademacher signs.  Order the $V_i$, and set
$w_i=b(V_i)^2$ and $r_i=a(V_i)^2/b(V_i)^2$.  The $r_i$ are then ordered in
the same direction.  For every upper set of indices $I$, cross multiplication
gives
\begin{equation}
\begin{split}
&\left(\sum_{i\in I}r_iw_i\right)\left(\sum_{j\notin I}w_j\right)
-\left(\sum_{i\in I}w_i\right)\left(\sum_{j\notin I}r_jw_j\right)\\
&\hspace{35mm}
=\sum_{i\in I}\sum_{j\notin I}w_iw_j(r_i-r_j)\ge0.
\end{split}
\label{eq:4.5}
\end{equation}
Thus the normalized squared $a$-weights majorize the normalized squared
$b$-weights; see \citet{marshall2011} for the majorization convention.

For completeness, the required Schur concavity of the Rademacher moments can
be seen from a two-coordinate equalization.  If $x,y\ge0$, $s=x+y$, and
$p=xy$, then
\begin{equation}
\begin{split}
\E(\varepsilon\sqrt x+\delta\sqrt y)^{2j}
&=\frac12\{(s+2\sqrt p)^j+(s-2\sqrt p)^j\}\\
&=\sum_{\ell=0}^{\lfloor j/2\rfloor}
\binom{j}{2\ell}s^{j-2\ell}(4p)^\ell.
\end{split}
\label{eq:4.6}
\end{equation}
At fixed $s$, this increases when $x$ and $y$ are made more equal.  Expanding
the full moment over the remaining signs expresses it as a nonnegative
combination of the pair moments in (\ref{eq:4.6}).  Hence the $2m$th moment of a
normalized Rademacher sum is Schur-concave in its squared coefficients, as in
\citet{efron1969}; see also \citet{eaton1970}. Combining this conditional comparison with (\ref{eq:4.5}), then
integrating over $V$, proves (\ref{eq:4.2}).

For $m\ge2$, a nontrivial equalization makes at least the $\ell=1$ term in
(\ref{eq:4.6}) strictly larger.  If $a/b$ is not almost everywhere constant, (\ref{eq:4.5})
is strict for some partial sum on an event of positive probability, which
proves the strict assertion.  Finally $|Z_F|\le\sqrt n$, so monotone termwise
summation proves (\ref{eq:4.3}).
\end{proof}

\begin{corollary}[uniform extremality for analytic transforms]
Let $F\in\mathcal F_{\su}$ and let $Z_U$ denote (\ref{eq:4.1}) for a centered-uniform
sample.  Then
\begin{equation}
\E Z_F^{2m}\le \E Z_U^{2m},
\qquad m\ge0,
\label{eq:4.7}
\end{equation}
and, for every $s,t\ge0$,
\begin{equation}
\E e^{sZ_F^2}\le \E e^{sZ_U^2},
\qquad
\E\cosh(tZ_F)\le \E\cosh(tZ_U).
\label{eq:4.8}
\end{equation}
If $F$ is not uniform up to scale, (\ref{eq:4.7}) is strict for $m\ge2$.
\end{corollary}

\begin{proof}
Let $G_F$ be the distribution function of $|X|$.  Symmetric unimodality makes
$G_F$ concave, so its quantile $a=G_F^{-1}$ is convex, nondecreasing, and
satisfies $a(0)=0$.  Hence $a(u)/u$ is nondecreasing.  The magnitude quantile
of a centered uniform law is proportional to $u$, and the theorem applies.
The two inequalities in (\ref{eq:4.8}) follow from their nonnegative even power
series.  Equality in all moments of order at least four requires $a$ to be
linear, which is uniformity up to scale.
\end{proof}

The indicator $\one_{\{x^2>r\}}$ has no representation of the form used in
(\ref{eq:4.3}).  Thus (\ref{eq:4.7})--(\ref{eq:4.8}) do not imply an ordering of the two-sided tails.
This fixed-dimension extremal statement within the symmetric unimodal class
is distinct from the broader transform bounds for self-normalized symmetric
sums in \citet{borisov2024}.

\section{Fixed-confidence dimension order}\label{sec:fixed-confidence}

The critical layer in Section~\ref{sec:dimension} is tied to
$\alpha_0$. At an arbitrary fixed confidence level, inversion of the same
expansion leads to a different dimension order. For $0<\alpha<1$ and
$m\ge2$, let $r_{m,\alpha}$ be the unique solution of
\begin{equation}
q_m(r_{m,\alpha})=\alpha,
\qquad
\rho_\alpha=\{\Phi^{-1}(1-\alpha/2)\}^2.
\label{eq:5.1}
\end{equation}
The directional density induced by the cube is positive on the sphere away
from lower-dimensional boundaries. Hence $q_m$ is continuous and strictly
decreasing on $(0,m)$, so the solution is unique.

\begin{theorem}[fixed-confidence expansion]
For every fixed $0<\alpha<1$,
\begin{equation}
r_{m,\alpha}
=\rho_\alpha-\frac{3\rho_\alpha(\rho_\alpha-3)}{10m}
+\frac{\rho_\alpha(1287-369\rho_\alpha-12\rho_\alpha^2)}{1400m^2}
+o(m^{-2}).
\label{eq:5.2}
\end{equation}
If $\rho_\alpha>3$, then
\begin{equation}
r_{m,\alpha}-r_{m-1,\alpha}
=\frac{3\rho_\alpha(\rho_\alpha-3)}{10m^2}+o(m^{-2})>0
\label{eq:5.3}
\end{equation}
for all sufficiently large $m$.  If $\rho_\alpha<3$, the eventual sign is
reversed.  At $\rho_\alpha=3$,
\begin{equation}
r_{m,\alpha_0}
=3+\frac{27}{175m^2}+\frac{d_{3,0}}{m^3}+o(m^{-3})
\label{eq:5.4}
\end{equation}
for a finite constant $d_{3,0}$, and
\begin{equation}
r_{m,\alpha_0}-r_{m-1,\alpha_0}
=-\frac{54}{175m^3}+o(m^{-3}).
\label{eq:5.5}
\end{equation}
\end{theorem}

\begin{proof}
The explicit second-order coefficient in (\ref{eq:3.24}) is
\begin{equation}
B(r)=\phi(\sqrt r)\sqrt r\,
\frac{63r^3-843r^2+1719r+873}{2800};
\label{eq:5.6}
\end{equation}
its derivation is recorded in Appendix~\ref{app:coefficients}.  Write
\[
r_{m,\alpha}=\rho_\alpha+d_1m^{-1}+d_2m^{-2}+o(m^{-2})
\]
and use the compact-uniform form of (\ref{eq:3.24}) on a neighborhood of
$\rho_\alpha$.  Substitution in
$N(r)+m^{-1}A(r)+m^{-2}B(r)+o(m^{-2})$, together with
$N(\rho_\alpha)=\alpha$, gives
\begin{equation}
d_1=-\frac3{10}\rho_\alpha(\rho_\alpha-3),
\qquad
d_2=\frac{\rho_\alpha(1287-369\rho_\alpha-12\rho_\alpha^2)}{1400}.
\label{eq:5.7}
\end{equation}
The remainder is uniform for the moving two-sided sets, so the root expansion
is valid.  Subtraction in dimensions $m$ and $m-1$ proves (\ref{eq:5.3}) and the
reversed case.  When $\rho_\alpha=3$, $d_1=0$ and $d_2=27/175$.  Retaining one
further term shows that the unknown third-order root coefficient contributes
only $O(m^{-4})$ to the successive difference, proving (\ref{eq:5.4})--(\ref{eq:5.5}).
\end{proof}

At the usual two-sided $5\%$ level,
\begin{equation}
\rho_{.05}=\{\Phi^{-1}(0.975)\}^2=3.841458820694\ldots,
\label{eq:5.8}
\end{equation}
so the uniform critical cone parameters are eventually increasing toward
$\rho_{.05}$. In particular, there is an integer $m_0$ such that, whenever
$n>j\ge m_0$,
\begin{equation}
q_j(r_{n,.05})<q_n(r_{n,.05})=.05.
\label{eq:5.9}
\end{equation}
This is an equal-scale dimension comparison.  General Khintchine mixtures
randomize all sample scales and are not ordered by (\ref{eq:5.9}); the fixed
$5\%$ calibration is therefore outside the conclusion of Section~\ref{sec:dimension}.

\appendix

\section{Coefficient calculations}\label{app:coefficients}

Let $\mu_j=\E V^j$, where $V=X^2$ for a variance-one uniform $X$.  Thus
\begin{equation}
\mu_1=1,
\qquad
\mu_2=\frac95,
\qquad
\mu_3=\frac{27}{7}.
\label{eq:A.1}
\end{equation}
With $D_n=\sum_i(V_i-1)$,
\begin{equation}
\E\sum_iV_i^2=n\mu_2,
\qquad
\E\left(\sum_iV_i^2D_n\right)=n(\mu_3-\mu_2),
\label{eq:A.2}
\end{equation}
and
\begin{equation}
\E\left(\sum_iV_i^2D_n^2\right)
=n^2\mu_2(\mu_2-1)+O(n).
\label{eq:A.3}
\end{equation}
Expansion of $(n+D_n)^{-2}$ gives
\begin{equation}
\E A_2
=\frac{\mu_2}{n}
+\frac{-2(\mu_3-\mu_2)+3\mu_2(\mu_2-1)}{n^2}
+O(n^{-3})
=\frac9{5n}+\frac{36}{175n^2}+O(n^{-3}).
\label{eq:A.4}
\end{equation}
Similarly,
\begin{equation}
\E A_3=\frac{27}{7n^2}+O(n^{-3}),
\qquad
\E(A_2^2)=\frac{81}{25n^2}+O(n^{-3}).
\label{eq:A.5}
\end{equation}
Write
\begin{equation}
\kappa_{4,n}=\frac{a_1}{n}+\frac{a_2}{n^2}+O(n^{-3}),
\qquad
\kappa_{6,n}=\frac{b_2}{n^2}+O(n^{-3}),
\label{eq:A.6}
\end{equation}
where
\begin{equation}
a_1=-\frac{18}{5},
\qquad
a_2=-\frac{72}{175},
\qquad
b_2=\frac{432}{7}.
\label{eq:A.7}
\end{equation}
The coefficient of $n^{-2}$ in (\ref{eq:3.21}) is therefore
\begin{equation}
B(r)=\phi(x)\left\{
-\frac6{175}H_3(x)
+\frac6{35}H_5(x)
+\frac9{400}H_7(x)
\right\},
\qquad x=\sqrt r.
\label{eq:A.8}
\end{equation}
Using
\begin{equation}
\begin{split}
H_3(x)&=x^3-3x,\\
H_5(x)&=x^5-10x^3+15x,\\
H_7(x)&=x^7-21x^5+105x^3-105x,
\end{split}
\label{eq:A.9}
\end{equation}
one obtains (\ref{eq:5.6}), and at $r=3$,
\begin{equation}
B(3)=\frac{9\sqrt3\phi(\sqrt3)}{175}.
\label{eq:A.10}
\end{equation}

For completeness, substitute
$r_{m,\alpha}=\rho+d_1m^{-1}+d_2m^{-2}+o(m^{-2})$ in
$N(r)+m^{-1}A(r)+m^{-2}B(r)$.  The coefficients of $m^{-1}$ and $m^{-2}$
are
\begin{equation}
N'(\rho)d_1+A(\rho)
\label{eq:A.11}
\end{equation}
and
\begin{equation}
N'(\rho)d_2+\frac12N''(\rho)d_1^2+A'(\rho)d_1+B(\rho),
\label{eq:A.12}
\end{equation}
respectively, where
\begin{equation}
N'(r)=-\frac{\phi(\sqrt r)}{\sqrt r},
\qquad
N''(r)=\frac{\phi(\sqrt r)(r+1)}{2r^{3/2}}.
\label{eq:A.13}
\end{equation}
Setting (\ref{eq:A.11})--(\ref{eq:A.12}) equal to zero yields (\ref{eq:5.7}).

\section{The case $n=6$}\label{app:n6}

At $n=6$ one can also obtain a finite comparison independent of the
asymptotic argument.  Let
\begin{equation}
r_*=\frac{61}{20},
\qquad
t_*^2=\frac{5r_*}{6-r_*}=\frac{305}{59}.
\label{eq:B.1}
\end{equation}
The calculation below proves
\begin{equation}
q_5(r_*)>\frac{51}{625},
\qquad
\frac{163}{2000}<q_6(r_*)<\frac{1631}{20000},
\label{eq:B.2}
\end{equation}
and hence
\begin{equation}
q_5(r_*)-q_6(r_*)>\frac{1}{20000},
\qquad
q_6(r_*)<0.0832<\alpha_0.
\label{eq:B.3}
\end{equation}
It also gives an explicit two-scale density with a strictly larger
six-observation tail.

Write $x=se+y$, where $e=n^{-1/2}\one$, $y\perp e$, $y=\rho u$, and
\begin{equation}
c_n=\sqrt{\frac r{n-r}}.
\label{eq:B.4}
\end{equation}
The cone condition is $|s|>c_n\rho$.  If
\[
A=\max_i u_i,
\qquad B=-\min_i u_i,
\qquad a=\sqrt n,
\]
integration first in $s$ and then in $\rho$ gives, on the positive half of
the cone,
\begin{equation}
I_n(A,B)=\frac{a^n}{n(n-1)}
\begin{cases}
(aA+c_n)^{-(n-1)},&aB-c_n\le aA+c_n,\\[3pt]
\displaystyle
\frac{2^n}{\{a(A+B)\}^{n-1}}-(aB-c_n)^{-(n-1)},
&aB-c_n>aA+c_n.
\end{cases}
\label{eq:B.5}
\end{equation}
The negative half gives the same integral after $u\mapsto-u$.

Partition the unit sphere in $\one^\perp$ into coordinate-order chambers.  On
one chamber use the projective section
\[
w=(0,y_1,\ldots,y_{n-2},1),
\qquad 0<y_1<\cdots<y_{n-2}<1,
\]
and put $v=w-\bar w\one$, $R=\lVert v\rVert$, and $u=v/R$.  Central
projection has Jacobian
\begin{equation}
d\sigma(u)=\frac{dy_1\cdots dy_{n-2}}{\sqrt n\,R^{n-1}}.
\label{eq:B.6}
\end{equation}
Moreover $A=\tau_n/R$.  Summing the $n!$ chambers and the two halves,
dividing by the cube volume, and replacing the ordered-simplex integral by a
unit-cube expectation gives, for $n=5,6$ and $r=r_*$,
\begin{equation}
q_n(r)=C_n\E\{\sqrt n\,\tau_n+c_nR_n\}^{-(n-1)},
\qquad
C_n=\frac{n^{(n-1)/2}}{2^{n-1}},
\label{eq:B.7}
\end{equation}
where $Y_1,\ldots,Y_{n-2}$ are i.i.d.\ uniform on $[0,1]$ and
\begin{equation}
R_n^2=1+\sum_jY_j^2-\frac{(1+\sum_jY_j)^2}{n},
\qquad
\tau_n=\frac{n-1-\sum_jY_j}{n}.
\label{eq:B.8}
\end{equation}
The second branch in (\ref{eq:B.5}) is absent in these two cases.  On the zero-sum unit
sphere,
\begin{equation}
B-A\le\frac{n-2}{\sqrt{n(n-1)}},
\label{eq:B.9}
\end{equation}
and the second branch would require $\sqrt n(B-A)>2c_n$.  It is therefore
excluded by
\begin{equation}
\frac{(n-2)^2}{n(n-1)}\le\frac{4r}{n(n-r)},
\qquad n=5,6.
\label{eq:B.10}
\end{equation}
At $r=61/20$, the two rational comparisons are
$9/20<244/195$ and $8/15<122/177$.

Put $z_j=Y_j-1/2$, $Z=\sum_jz_j$, and
\begin{equation}
W=\sum_jz_j^2-\frac{Z^2}{n},
\qquad
b_n(W)=\frac n2+\lambda_n\sqrt{\frac12+W},
\qquad
\lambda_n^2=\frac{nr}{n-r}.
\label{eq:B.11}
\end{equation}
Expanding the chamber denominator gives the nonnegative series
\begin{equation}
q_n(r)=\frac{n^{n-1}}{2^{n-1}}
\sum_{k=0}^\infty \binom{n+2k-2}{2k}
\E\{Z^{2k}b_n(W)^{-(n-1+2k)}\}.
\label{eq:B.12}
\end{equation}
Indeed,
$\sqrt n\,\tau_n+c_nR_n=(b_n(W)-Z)/\sqrt n$ and
$|Z|<(n/2)\le b_n(W)$; symmetry removes the odd powers.

Let
\begin{equation}
q_c=\frac{n+2}{8},
\quad
w_c=\frac{n-2}{8},
\quad
\nu=\frac{W-w_c}{q_c},
\quad
\rho_n=\frac{n-2}{n+2}.
\label{eq:B.13}
\end{equation}
For each $s\ge1$ there are nonnegative coefficients $c_{s,j}$ such that
\begin{equation}
b_n(W)^{-s}=A_c^{-s}\sum_{j=0}^\infty c_{s,j}(-\nu)^j,
\qquad
\sum_{j=0}^\infty c_{s,j}z^j
=\{1-\beta(1-\sqrt{1-z})\}^{-s},
\label{eq:B.14}
\end{equation}
where $A_c=n/2+\lambda_n\sqrt{q_c}$ and
$\beta=(A_c-n/2)/A_c$.  Since $|\nu|\le\rho_n<1$,
\begin{equation}
\left|b_n(W)^{-s}-A_c^{-s}\sum_{j=0}^Jc_{s,j}(-\nu)^j\right|
\le\left(\frac n2\right)^{-s}\rho_n^{J+1}.
\label{eq:B.15}
\end{equation}
The coefficients are rational combinations of
\begin{equation}
a_\ell=[z^\ell](1-\sqrt{1-z})
=\frac{\binom{2\ell}{\ell}}{(2\ell-1)4^\ell},
\qquad \ell\ge1.
\label{eq:B.16}
\end{equation}
All mixed moments in the truncated expression are rational.  If
\begin{equation}
M_m(a,b)=\E\left\{
\left(\sum_{i=1}^mz_i\right)^a
\left(\sum_{i=1}^mz_i^2\right)^b\right\},
\label{eq:B.17}
\end{equation}
independence gives the recursion
\begin{equation}
M_m(a,b)=\sum_{i=0}^a\sum_{j=0}^b
\binom{a}{i}\binom{b}{j}M_{m-1}(i,j)\,
\mu_{a-i+2(b-j)},
\label{eq:B.18}
\end{equation}
where
\begin{equation}
M_0(a,b)=\one_{\{a=b=0\}},
\qquad
\mu_{2h}=\frac1{(2h+1)2^{2h}},
\qquad
\mu_{2h+1}=0.
\label{eq:B.19}
\end{equation}
Expanding $W=\sum z_i^2-Z^2/n$ in (\ref{eq:B.18}) gives every moment in (\ref{eq:B.12}).

For the positive $k$-tail in (\ref{eq:B.12}), put
\begin{equation}
B_k=\binom{n+2k-2}{2k},
\qquad
b_{\min}=\frac n2+\frac{\lambda_n}{\sqrt2},
\qquad
v=\frac{n-2}{2b_{\min}},
\label{eq:B.20}
\end{equation}
and
\begin{equation}
\gamma_L=
\frac{(n-1+2L)(n+2L)}{(2L+1)(2L+2)}.
\label{eq:B.21}
\end{equation}
Because $B_{k+1}/B_k$ decreases with $k$, the tail beginning at $L$ is at
most
\begin{equation}
\frac{n^{n-1}}{2^{n-1}}b_{\min}^{-(n-1)}
\frac{B_Lv^{2L}}{1-\gamma_Lv^2},
\qquad \gamma_Lv^2<1.
\label{eq:B.22}
\end{equation}
Taking $J=24$, retaining $0\le k\le15$, and using $L=16$ gives (\ref{eq:B.2})--(\ref{eq:B.3}).
The resulting certificate is displayed in Table~\ref{tab:finite-certificate}.

\begin{table}[!b]
\caption{Finite enclosures in the calculation at $r_*=61/20$. The two error
columns bound, respectively, truncation of (\ref{eq:B.14}) after $J=24$ and the
positive $k$-tail in (\ref{eq:B.12}) beginning at $L=16$.}
\label{tab:finite-certificate}
\centering
\small
\begin{tabular}{@{}ccccc@{}}
\hline
$n$ & Truncated sum & $J$-error & $k$-tail & Certified enclosure for $q_n(r_*)$ \\
\hline
$5$ & $0.08164905813136$ & $1.029\cdot10^{-9}$ & $4.652\cdot10^{-13}$
& $[0.08164905710,\,0.08164905917]$ \\
$6$ & $0.08154616628057$ & $6.035\cdot10^{-8}$ & $6.636\cdot10^{-9}$
& $[0.08154610593,\,0.08154623327]$ \\
\hline
\end{tabular}
\end{table}

The displayed endpoints are rounded outward and absorb the outward dyadic
rounding of all square roots. Every square root used to form the truncated
sums and bounds was enclosed by consecutive dyadic rationals, with the
ordering checked by integer squaring. The enclosures in
Table~\ref{tab:finite-certificate} use denominator $2^{48}$; denominator
$2^{16}$ already suffices for the coarser rational comparisons in
(\ref{eq:B.2})--(\ref{eq:B.3}).

It remains to compare with $\alpha_0$.  Let
$R(x)=\bar\Phi(x)/\phi(x)$ be the Mills ratio.  By the alternating bounds for
Laplace's continued fraction in \citet[(3.1)]{lee1992}, the odd convergent
$L_{11}$ is a lower bound, and at $x=\sqrt3$ it equals
\begin{equation}
R(\sqrt3)>
\cfrac1{\sqrt3+\cfrac1{\sqrt3+\cfrac2{\sqrt3+\cdots+
\cfrac{11}{\sqrt3}}}}
=\frac{4412\sqrt3}{16341}.
\label{eq:B.23}
\end{equation}
Together with $\pi<355/113$ and the rational Taylor upper bound
\begin{equation}
e^{3/2}<
\sum_{j=0}^{30}\frac{(3/2)^j}{j!}
+\frac{(3/2)^{31}/31!}{1-3/64},
\label{eq:B.24}
\end{equation}
this yields
\begin{equation}
\alpha_0=2\phi(\sqrt3)R(\sqrt3)
>\frac{333}{4000}>\frac{52}{625}.
\label{eq:B.25}
\end{equation}

Finally let $\varepsilon=10^{-14}$ and $p=10^{-5}$, and use the density
$f_{p,\varepsilon}$ in (\ref{eq:2.8}).  If one scale is $\varepsilon$ and five are one,
write the resulting cone probability as $Q_\varepsilon$.  Changing the small
coordinate to zero perturbs the five-dimensional cone polynomial by at most
\begin{equation}
\eta=10\varepsilon+(r_*-1)\varepsilon^2.
\label{eq:B.26}
\end{equation}
Conditioning on four unit-scale coordinates and bounding the interval on
which the remaining quadratic lies within $\eta$ of zero gives
\begin{equation}
Q_\varepsilon-q_6(r_*)
\ge q_5(r_*)-q_6(r_*)-\sqrt{\frac{2\eta}{r_*-1}}
>\frac1{25000},
\label{eq:B.27}
\end{equation}
since $\sqrt{2\eta/(r_*-1)}<1/(3\cdot10^6)$.
The probability of two or more small scales is at most $\binom{6}{2}p^2$.
Conditioning on the number of small scales therefore gives
\begin{equation}
\begin{split}
\Pp_{f_{p,\varepsilon}}\{|T_6|>t_*\}
-\Pp_U\{|T_6|>t_*\}
&>\frac{6p(1-p)^5}{25000}-15p^2\\
&>\frac1{2\cdot10^9}.
\end{split}
\label{eq:B.28}
\end{equation}
Convolution with a sufficiently narrow compactly supported $C^\infty$
symmetric decreasing approximate identity preserves the strict inequality.

\section*{Acknowledgments}
\textbf{AI tool disclosure.}
OpenAI's ChatGPT was used for literature-search support, additional checks of
selected derivations and the finite calculation in Appendix~\ref{app:n6}, and
language and LaTeX editing. The research question, mathematical results, and
proofs were developed by the author. The author reviewed all AI-assisted
material incorporated into the manuscript and assumes full responsibility
for its correctness, originality, and attribution.

\end{document}